\documentclass[12pt,a4paper,leqno]{amsart}

\usepackage{ebgaramond}

\usepackage{amssymb}
\usepackage{stmaryrd}
\usepackage[mathscr]{eucal}
\usepackage{latexsym}        % \Box \Diamond \lhd \leadsto \mho ... (à retirer si inutilisé)
\usepackage{extarrows}
\usepackage{amscd}           % environnement CD (à retirer si tu n'utilises que xy)
\everymath{\displaystyle}

\usepackage[all]{xy}
\usepackage{graphicx}
\usepackage{booktabs}
\usepackage{xcolor}
\usepackage{stackengine}
\usepackage{xspace}

\usepackage{hyperref}

\AtBeginDocument{\setbox0=\hbox{$\scriptscriptstyle x$}}

\newcommand{\bbmu}{%
  \stackengine{0pt}{\mu}{\kern0.15ex\mu}{O}{l}{F}{T}{S}%
}

\theoremstyle{plain}
\newtheorem{theorem}{Theorem}[section]

\newtheorem{proposition}[theorem]{Proposition}
\newtheorem{corollary}[theorem]{Corollary}

\newtheorem{conjecture}[theorem]{Conjecture}

\theoremstyle{definition}

\newtheorem{definition}[theorem]{Definition}

\newtheorem{remark}[theorem]{Remark}

\title{On a multiplicative perturbation of Laguerre polynomials}
\author{Julien Grivaux}
\address{Sorbonne Université, Université Paris Cité, CNRS, IMJ-PRG, F-75005 Paris, France}
\email{julien.grivaux@imj-prg.fr}

\begin{document}
\begin{abstract}
In this article, we study the family of polynomials \[P_n(s, z)=\sum_{k=0}^n s^{(k)} z^{n-k}\] for $s>0$. We prove that its roots are simple, and provide a precise localisation of them in specific angular sectors.
\end{abstract}
\maketitle
\tableofcontents
\section{Introduction}
The zeroes of hypergeometric polynomials 
\[
{}_2F_1(-n,b;c;z)=\sum_{k=0}^n (-1)^k\binom{n}{k} \frac{b^{(k)}}{c^{(k)}} z^k
\] 
for real parameters $b,c$ have been the object of intense investigations at the end of the $19^{\mathrm{th}}$ century. Three major mathematicians, Klein \cite{Klein1890}, Hilbert \cite{Hilbert1888} and Hurwitz \cite{Hurwitz1891} arrived at essentially the same result (today called the Hilbert-Klein formula) by different means: the precise counting of roots of ${}_2F_1(-n,b;c;z)$ in each of the three segments $]-\infty, 0[$, $]0, 1[$ and $]1; + \infty[$. 
These polynomials are up to a change of variable Jacobi polynomials
\[
{}_2F_1(-n,b;c;z)=\dfrac{n!}{c^{(n)}}\,P_n^{(c-1,\,b-n-c)}(1-2z).
\]
and the Hilbert-Klein formula in this setting is \cite[Thm. 6.72]{Szego1975}. Recall that the Jacobi polynomials are orthogonal on $]-1, 1[$ for the weight $(1-x)^{\alpha}(1+x)^{\beta}$ if $\alpha, \beta >-1$. It follows that on the region $c>0$ and $b>n+c-1$, all roots of ${}_2F_1(-n,b;c;X)$ are real. The hypergeometric differential equation
\[
z(1-z)y''+[c-(b+1-n)z]y' + nb y=0
\]
implies that the only multiple root that can occur is at $z=1$, and it follows from the Chu-Vandermonde identity  that this happens if and only if $b-c$ belongs to $\llbracket 0, n-2 \rrbracket$. This is in accordance with Hilbert's formula for the discriminant of these polynomials given in 
\textit{loc. cit} (\textit{see} \cite[Th 6.71, eq. (6.71.5)]{Szego1975}). 
A confluent version of the preceding story is given by the polynomials
\[
{}_1F_1(-n;s;z)=\sum_{k=0}^n (-1)^k\binom{n}{k} \frac{z^{k}}{s^{(k)}}
\] 
which are linked to generalized Laguerre polynomials via the formula
\[
{}_1F_1(-n;s;z)=\dfrac{n!}{s^{(n)}}\,L_n^{(s-1)}(z).
\]
The differential equation is Kummer's equation
\[
zy''+(s-z)y'+ny=0,
\]
it implies that there is no multiple roots at all. The discriminant of generalized Laguerre polynomials had been computed by Schur \cite{Schur1931} (\textit{see} \cite[Th 6.71, eq. (6.71.6)]{Szego1975}), and the number of real roots of this polynomial when $s$ is real is explicitly described in \cite[Thm. 6.73]{Szego1975}. If $s>0$, these polynomials are orthogonal on $\mathbb{R}_{+}$ for the measure $t^{s-1}e^{-t}dt$ so all their roots are all real. 
\par \medskip
In the present paper, we are interested in the polynomials
\[
P_n(s,z)=\sum_{k=0}^{n} s^{(k)} z^{n-k}
\]
when the parameter $s$ is positive. They satisfy the convolution equation
\[
P_n(s, z) \boxtimes_n {}_1 F_1(-n; s; -z)=1+z+ \ldots + z^n
\]
where $\boxtimes_n$ is Schur-Szeg\H{o} convolution (\textit{see} \cite[Chap.5 \S 5.5]{Rahman}). Its roots are not real at all for $s>0$ because as $s$ goes to infinity, a Puiseux expansion proves that the roots of $P_n(s, z)$ grow asymptotically as $\zeta s$ where $\zeta$ runs through $n+1$-th roots of unity apart from $1$. Second, this polynomial can have nontrivial multiple roots, the simplest example being
\[
P_2(-4/3,z)=(z-2/3)^2.
\]
There is no simple available expression for the discriminant of $P_n(s,z)$, and the differential equation satisfied by $P_n(s, z)$ is
\[
(z-n-s)y+zy'=z^{n+1}-s^{(n+1)}.
\]
Although it is not homogeneous, it gives some control of the multiple roots of $P_n(s,z)$. Our main result shows that the region $s>0$ is still a region where the roots of $P_n(s,z)$ are simple and can be localised. More precisely, 
\begin{theorem} \label{main}
Let $s$ be a positive real number. 
\begin{enumerate}
\item[(i)] If $n$ is even, then $P_n(s,z)$ has one root in any of the $n+1$-regular open angular sectors, except the one containing the negative real line.
\item[(ii)] If $n$ is odd, then $P_n(s,z)$ has one negative real root, which is in the interval $]-\left(s^{(n+1)}\right)^{1/n+1}, 0[$, and one root in any of the $n+1$-regular open angular sectors except the two adjacent to the negative real line.
\end{enumerate}
In particular, the roots of $P_n(s,z)$ are simple.
\end{theorem} 
Below is a Maple drawing illustrating the theorem for $n=13$.
\par \medskip
\begin{figure}[ht]
  \includegraphics[width=0.8\linewidth]{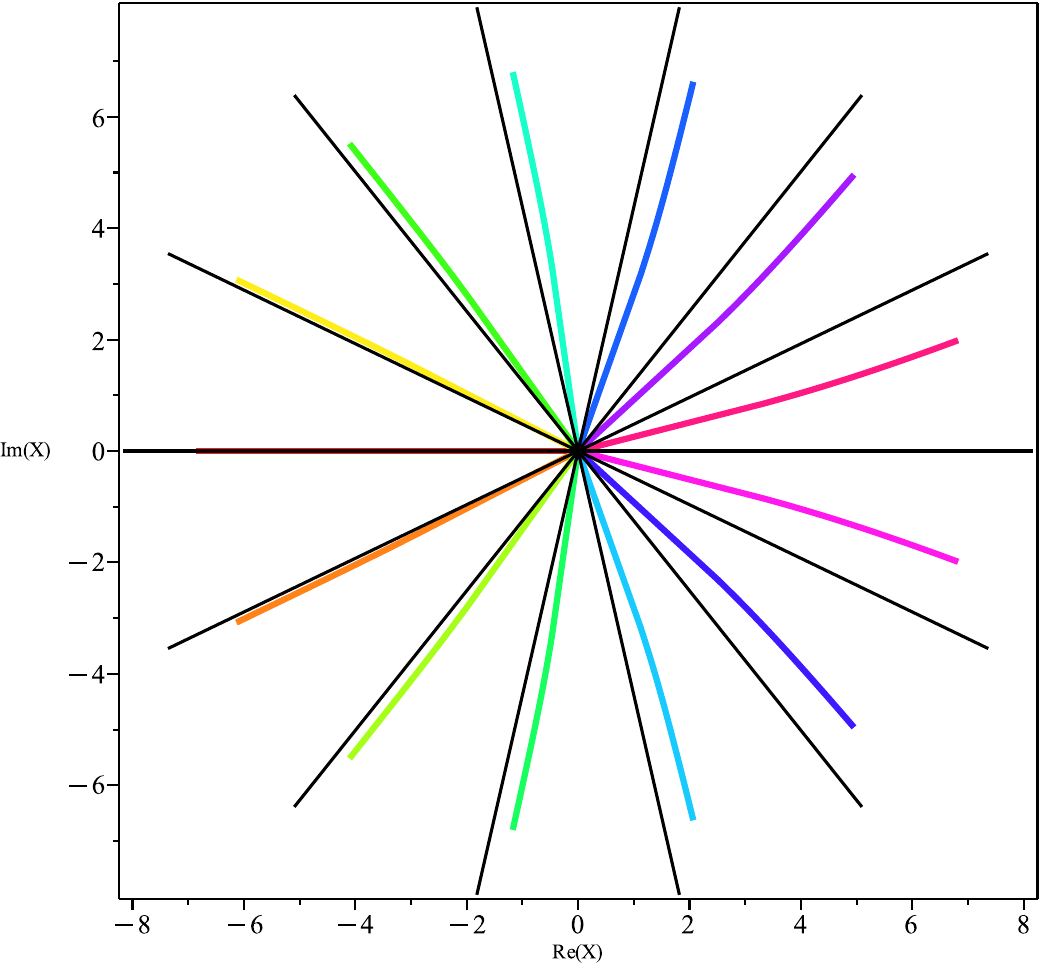}
  \caption{Roots of $P_{13}(s, z)$ for $s$ in $[0,2]$.}
  \label{fig:racines}
\end{figure}

\section{Algebraic identities}
For any complex number $s$ and any nonnegative integer $k$, the ascending factorial $s^{(k)}$ is defined as $s^{(k)}=1$ if $k=0$, and 
\[
s^{(k)}=s(s+1) \ldots (s+k-1)
\] 
otherwise.

\begin{definition}
For any positive integer $n$, the bivariate polynomial $P_n(s,z)$ is defined by $P_n(s,z)=\sum_{k=0}^{n} s^{(k)} z^{n-k}$.
\end{definition}

We state and prove elementary identities relating the polynomials $P_n(s,z)$.
\begin{proposition} \label{basic}
The following identities hold: 
\begin{enumerate}
\item[(i)] $sP_{n}(s+1, z)=(s+n)P_n(s,z)-z \partial_z P_n(s,z)$.
\item[(ii)] $sP_n(s+1,z)=zP_n(s,z)+s^{(n+1)}-z^{n+1}$.
\item[(iii)] $(z-s-n)P_n(s, z)+z \partial_z P_n(s,z)=z^{n+1}-s^{(n+1)}$.
\end{enumerate}
\end{proposition}

\begin{proof}
For (i), each monomial $s^{(k)}z^{n-k}$ satisfies the identity. Indeed:
\begin{align*}
(s+n)s^{(k)}z^{n-k} -z \left((n-k) s^{(k)} z^{n-k-1}\right) 
&= (s+k) s^{(k)} z^{n-k}\\
&= s^{(k+1)} z^{n-k} \\
&= s (s+1)^{(k)} z^{n-k}.
\end{align*}
The second identity is straightforward: 
\begin{align*}
zP_n(s, z)+s^{(n+1)}-z^{n+1} &= \sum_{k=1}^{n} s^{(k)} z^{n-(k-1)}+s^{(n+1)} \\
&= \sum_{\ell=0}^{n} s^{(\ell+1)} z^{n-\ell} \\
&= sP_{n}(s+1, z).
\end{align*}
The third identity is obtained by equating the two right members of (i) and (ii).
\end{proof}

\begin{corollary} \label{yolo}
Let $\Delta_n(s)$ be the discriminant of $P_n(s,z)$. Then:
\begin{enumerate}
\item[(i)] $\mathrm{res}_z(z^{n+1}-s^{(n+1)}, P_n(s, z)) = (-1)^{n(n+1)/2} s^{(n)} \Delta_n(s)$.
\item[(ii)] $s^n \mathrm{res}_z(P_n(s, z),P_n(s+1,z)) =  (-1)^{n(n-1)/2}s^{(n)} \Delta_n(s)$.
\end{enumerate} 
\end{corollary}

\begin{proof}
For (i), we compute:
\begin{align*}
&\mathrm{res}_z(z^{n+1}-s^{(n+1)}, P_n(s, z)) \\
&=\mathrm{res}_z((z-s-n)P_n(s, z)+z \partial_z P_n(s,z), P_n(s,z)) \\
&=\mathrm{res}_z(P_n(s,z),z \partial_z P_n(s, z)) \\
&=\mathrm{res}_z(P_n(s,z),z) \times \mathrm{res}_z(P_n(s,z), \partial_z P_n(s, z)) \\
&=(-1)^n P_n(s,0) \times (-1)^{n(n-1)/2} \Delta_n(s) \\
&=(-1)^{n(n+1)/2} s^{(n)} \Delta_n(s).
\end{align*}
Similarly, for (ii),
\begin{align*}
&s^n \mathrm{res}_z(P_n(s, z),P_n(s+1,z)) \\
&=\mathrm{res}_z(sP_n(s+1,z),P_n(s, z)) \\
&= \mathrm{res}_z(z P_n(s,z)+s^{(n+1)}-z^{n+1}, P_n(s,z)) \\
&=\mathrm{res}_z(s^{(n+1)}-z^{n+1}, P_n(s,z)) \\
&=(-1)^n \mathrm{res}_z(z^{n+1}-s^{(n+1)}, P_n(s,z)) \\
&=(-1)^{n(n-1)/2} s^{(n)} \Delta_n(s) \qquad \textrm{by} \quad \textrm{(i)}. 
\end{align*}
\end{proof}

\section{Determinantal expression}
In this section, we give a general expression of the discriminant $\Delta_n(s)$ in terms of a specific determinant. To do this, we place ourselves in a general framework. 
\par \medskip
For $a=(a_1, \ldots, a_{n+1}) \in \mathbb{C}^{n+1}$, we put
\begin{align*}
 M(a) &= \begin{vmatrix}
1      & a_1     & a_1 a_2     & \ldots & a_1 a_2 \ldots a_n \\
1      & a_2     & a_2 a_3     & \ldots & a_2 a_3 \ldots a_{n+1} \\
\vdots & \vdots  & \vdots      &        & \vdots \\
1      & a_{n+1} & a_{n+1} a_1 & \ldots & a_{n+1} a_1 \ldots a_{n-1}
\end{vmatrix}\\
C_a&=a_1 \ldots a_{n+1} \\
P(z) &= z^n + a_1 z^{n-1} + \ldots + a_{n-1}z + a_{n} \in \mathbb{C}[z] \\
Q(t) &=\mathrm{res}_z(z^{n+1}-t, P(z)) \in \mathbb{C}[t]. 
\end{align*}

\begin{proposition}
$Q(C_a)=(-1)^{n(n+1)/2} a_1^n a_2^{n-1} \ldots a_n \times M(a)$.
\end{proposition}

\begin{proof}
We write $Q(C_a)$ as the determinant of the multiplication by the polynomial $P$ on $\mathbb{C}[z]/(z^{n+1}-C_a)$, which gives
\[
\begin{vmatrix}
a_1 \ldots a_n & C_a & a_1 C_a & \ldots \, & \\
a_1 \ldots a_{n-1} & a_1 \ldots a_n & C_a & \ldots &  \\
\vdots & \vdots & a_1 \ldots a_{n-1} & \ldots &  \\
\vdots & \vdots & \ldots & & \\
a_1 & \vdots & \ldots & & \\
1 & a_1 & a_1 a_2 & \ldots &
\end{vmatrix}
\]
After factorizing $a_1$ in the second column, $a_1 a_2$ in the third column and so on, taking the transpose, and reordering the lines from bottom to top, we get the required expression. 
\end{proof}

\begin{definition}
For any positive integer $n$ and any complex number $s$, we put $\theta_n(s)=M(s, s+1, \ldots, s+n)$.
\end{definition}

\begin{proposition} \label{malin} The two following identities hold: 
\begin{enumerate}
\item[(i)] $\Delta_n(s)=s^{(n-1)} s^{(n-2)} \ldots s^{(1)}\, \theta_n(s)$.
\item[(ii)] $\mathrm{res}_z(z^{n+1}-s^{(n+1)}, P_n(s, z))= (-1)^{n(n+1)/2} s^{(n)} s^{(n-1)} \ldots s^{(1)}\, \theta_n(s)$.
\end{enumerate}
\end{proposition}

\begin{proof} 
Thanks to corollary \ref{yolo} (i), since $s^{(n+1)}=C(s, s+1, \ldots, s+n)$, we have $P(z)=P_n(s, z)$ so
\begin{align*}
(-1)^{n(n+1)/2} & s^{(n)} \Delta_n(s) = \mathrm{res}_z(z^{n+1}-s^{(n+1)}, P_n(s, z)) \\
&= Q(C(s, s+1, \ldots, s+n)) \\
&= (-1)^{n(n+1)/2} s^{(n)} s^{(n-1)} \ldots s^{(1)} M(s, s+1, \ldots, s+n)
\end{align*}
whence the result.
\end{proof}

\section{Real roots}

\begin{proposition} \label{locus}
If $s > 0$, $P_n(s, z)$ has no real root if $n$ is even, and has one real root lying in $]-\left(s^{(n+1)}\right)^{1/n+1}, 0[$ if $n$ is odd. 
\end{proposition}

\begin{proof}
Let us consider the differential equation 
\begin{equation} \label{1}
(x-n-s) P(x) + x P'(x)=x^{n+1}-s^{(n+1)}
\end{equation}
satisfied by $P_n(s, z)$.
\par \medskip
Assume that $n$ is even. If $\alpha$ realizes the minimum of the polynomial $P$ on $\mathbb{R}$, we may assume that $\alpha <0$ but then $(\alpha-n-s)P(\alpha)=\alpha^{n+1}-s^{(n+1)}<0$ by $\eqref{1}$. Since $\alpha-n-s <0$, $P(\alpha)>0$, hence $P >0$ on $\mathbb{R}$.
\par \medskip
Assume that $n$ is odd. If $\alpha$ is the smallest root of $P$, then $P'(\alpha) \geq 0$. If $P'(\alpha)=0$, $P''(\alpha)=(n+1)\alpha^{n-1}>0$ which is absurd, so $P'(\alpha)>0$. If follows again from \eqref{1} that $\alpha \in ]-\left(s^{(n+1)}\right)^{1/n+1}, 0[$. Besides, it is the only root, because otherwise there would exist a root $\beta$ in $]\alpha, 0[$ such that $P'(\beta) \leq 0$, which contradicts again \eqref{1}.
\end{proof}

\section{Integral representation}

For $\mathrm{Re}(s)>0$, $s^{(k)}=\frac{\Gamma(s+k)}{\Gamma(s)}=\frac{1}{\Gamma(s)} \int_0^{+\infty} t^{k+s-1}e^{-t}dt$, so
\[
P_n(s, z)=\frac{1}{\Gamma(s)} \int_{0}^{+\infty} \frac{t^{n+1}-z^{n+1}}{t-z} t^{s-1}e^{-t}dt.
\] 

\begin{remark}
We can give another proof of \ref{basic} (ii) with the integral representation. Indeed, 
\begin{align*}
sP_n(s+1, z)-zP_n(z)&= \frac{1}{\Gamma(s)} \int_{0}^{+\infty} (t^{n+1}-z^{n+1})\, t^{s-1}e^{-t}dt \\
&=\frac{\Gamma(s+n+1)}{\Gamma(s)}-z^{n+1} \\
&=s^{(n+1)}-z^{n+1}.
\end{align*}
\end{remark}

\begin{proposition} \label{zendan}
If $s>0$ and $P(s, \alpha)=0$, then either $\alpha \in \mathbb{R}_-$ and $n$ is odd, or $\alpha^{n+1} \notin \mathbb{R}_{+}$.
\end{proposition}

\begin{proof}
The root $\alpha$ cannot be positive. Assume that $\mathrm{Im}\,(\alpha) >0$. If $t$ goes through the real line, the curve $t \mapsto \frac{1}{t-\alpha}$ travels through the circle 
\[
\mathscr{C}\left(\frac{\mathbf{i}}{2\mathrm{Im}\,(\alpha)}, \frac{1}{ 2\mathrm{Im}\,(\alpha)}\right),
\] 
originating and terminating at $0$. Let $\ell$ be any nonzero real linear form of the complex plane that vanishes on $\frac{1}{|\alpha|-\alpha}\cdot$ As $\left[0, \frac{1}{|\alpha|-\alpha}\right]$ is a chord on the previous circle, we can assume that 
\[
\begin{cases}
\ell \left(\frac{1}{t-\alpha}\right)<0 \quad \textrm{if} \quad t < |\alpha| \vspace{0.2cm} \\
\ell \left(\frac{1}{t-\alpha}\right)>0 \quad \textrm{if}  \quad t > |\alpha|.
\end{cases}
\]
In particular, for all real $t \neq |\alpha|$, we have
\[
\ell \left(\frac{t^{n+1}-\alpha^{n+1}}{t-\alpha}\right)>0.
\]
Hence, $\ell(P_n(s, \alpha))>0$.
\end{proof}

\begin{remark}
With the integral representation, we can see directly that $P_n(s, z)$ is positive if $s>0$, $z \in \mathbb{R}$ and $n$ is even. Indeed, for $z$ real and $t \neq z$,
\[
\frac{t^{n+1}-z^{n+1}}{t-z} >0.
\]
\end{remark}

\begin{corollary}
If $s>0$, the roots of $P_n(s,z)$ are simple.
\end{corollary}

\begin{proof}
Thanks to corollary \ref{yolo} (ii), it suffices to prove that $P_n(s, \alpha) $ is nonzero if $\alpha^{n+1}=s^{(n+1)}$. In this case, the only solution is that $\alpha$ be in $\mathbb{R}_-$, which is impossible thanks to proposition \ref{locus}.
\end{proof}

\begin{corollary}
If $s\geq 0$, $(-1)^{n(n-1)/2} \theta_n(s) >0$. 
\end{corollary}

\begin{proof}
Thanks to proposition \ref{malin} (i), if $s \neq 0$, $\theta_n(s) \neq 0$. Besides, we can make an asymptotic expansion: as $s$ goes to $0$, the roots of $P_n(s, z)$ are asymptotically $(-s(n-1)!)^{1/n}$ so 
\begin{align*}
\mathrm{res}_z (z^{n+1}- s^{(n+1)}, & P_n(s, z))  \\
&\sim_{s \to 0^+} \prod_{\zeta^n=-1}  (-\zeta (s(n-1)!)^{n+1/n}-s^{(n+1)}) \\
&  \sim_{s \to 0^+}  (-1)^{n} (n!)^{n} s^{n}
\end{align*}
but on the other hand, by proposition \ref{malin} (ii), 
\[
\mathrm{res}_z (z^{n+1}- s^{(n+1)},  P_n(s, z)) \sim_{s \to 0^+} (-1)^{n(n+1)/2} s^n \times\prod_{k=1}^{n-1} k! \times \theta_n(s)
\]
so $\theta_n(0) = (-1)^{n(n-1)/2} \frac{(n!)^n}{\prod_{k=1}^{n-1} k!} \cdot$ The result follows.
\end{proof}

\begin{remark}
Proving that $(-1)^{n(n-1)/2} \theta_n(s) \geq 0$ is completely elementary. Indeed, 
\[
\mathrm{res}_z (z^{n+1}- s^{(n+1)},  P_n(s, z)) = \prod_{\alpha^{n+1}=s^{(n+1)}} P_n(s, \alpha)
\]
and all pairs of complex conjugate $(n+1)$-th roots of $s^{(n+1)}$ produce conjugate terms in the product, hence a positive contribution. It remains to look at possible contributions when $\alpha$ is real. If $\alpha>0$, then $P_n(s, \alpha)>0$. The case $\alpha<0$ occurs when $n$ is odd, then $P_n(s, \alpha)<0$ thanks to proposition \ref{locus}. Hence $(-1)^n \mathrm{res}_z (z^{n+1}- s^{(n+1),  P_n(s, z)}) \geq 0$ and we conclude using proposition \ref{malin} (ii).
\end{remark}

\section{Angular sectors}
For any positive integer $p$, the $p$-regular open angular sectors are the open angular sectors delimited by two half lines originating at $0$ and passing by two successive $p$-roots of unity.
\begin{theorem}
If $s>0$ then:
\begin{enumerate}
\item[(i)] If $n$ is even, then $P_n(s,z)$ has one root in any of the $n+1$-regular open angular sectors, except the one containing the negative real line.
\item[(ii)] If $n$ is odd, then $P_n(s,z)$ has one negative real root and one root in any of the $n+1$-regular open angular sectors except the two adjacent to the negative real line.
\end{enumerate}
\end{theorem}

\begin{proof}
Thanks to proposition \ref{zendan} and the continuity of roots of $P(s, z)$, it suffices to prove the result for one value of $s$, so we can argue asymptotically as  $s$ goes to  $0^{+}$. In this case, the $n$ roots of $z \mapsto P_m(s,z)$ are asymptotically the roots of the equation $z^n+(n-1)! s=0$, that is $z \sim (-s(n-1)!)^{1/n} $. Hence the arguments of the $n$ roots of $P_n(s, z)$ are asymptotically the arguments of the $n$-th roots of $-1$. We distinguish the two situations, $n$ even and $n$ odd.
\par \medskip
If $n$ is even, $n=2k$, then the $2k$-th roots of $-1$ have argument $\frac{(2\ell+1) \pi}{2k}$ for $-k \leq \ell \leq k-1$. Since
\[
\frac{(2\ell+1) \pi}{2k}=\frac{2\ell \pi}{2k+1} + \underbrace{\frac{2 \ell+2k+1}{4k}}_{\textrm{in}\,]0,1[} \times \frac{2\pi}{2k+1}
\]
we see that these roots lie in the $2k$ distinct $2k+1$-regular open angular sectors 
\[
\frac{2\ell \pi}{2k+1} < \mathrm{Arg} (z) < \frac{2(\ell+1) \pi}{2k+1} \cdot
\]
where $-k \leq \ell \leq k-1$. The only missing angular sector is for $\ell=k$, that is
\[
\frac{2k \pi}{2k+1} < \mathrm{Arg} (z) < \frac{(2k+2) \pi}{2k+1},
\]
it is exactly the one containing the negative real line.
\par \medskip
If $n$ is odd, $n=2k+1$, then the $2k+1$-th roots of $-1$ have argument  $\frac{(2\ell+1) \pi}{2k+1}$ for $-k \leq \ell \leq k$. As before, we write
\[
\frac{(2\ell+1) \pi}{2k+1}=\frac{2\ell \pi}{2k+2} +  \underbrace{\frac{k+\ell+1}{2k+1}}_{\textrm{in}\,]0,1[ \, \textrm{if}\, \ell \neq k} \times \frac{2\pi}{2k+2} \cdot
\]
Hence the first $2k$ among $2k+1$-th roots of $-1$ are in $2k$ distinct $2k+1$-regular open angular sectors. The two remaining angular sectors are
\[
\begin{cases}
\frac{(2k+1) \pi}{2k+2} < \mathrm{Arg} (z) < \pi \vspace{0.2cm}\\
\pi < \mathrm{Arg} (z) < \frac{(2k+3) \pi}{2k+2}
\end{cases}
\]
which are the two angular sectors bordering the negative real line. Since we know that $P_n$ has one negative real root, the result follows.  
\end{proof}

\section{The stability conjecture}
In this section, we explain why the present work leads naturally to a conjecture involving Hurwitz stability of some resultants. Let us give the first values of the polynomials $\theta_n(s)$:

\begin{align*}
\theta_2(s)&=-3s-4 \\
\theta_3(s)&= -16 s^{3}-96 s^{2}-184 s-108 \\
\theta_4(s)&= 125 s^{6}+2000 s^{5}+13025 s^{4}+43950 s^{3}+80500 \\
&\qquad + s^{2}+75240 s +27648
\end{align*}

Numerical calculations suggest that the coefficients of $(-1)^{n(n-1)/2} \theta_n(s)$ are all positive. Proving it would give a much simpler proof that $\theta_n(s)$ does not vanish for $s \geq 0$, but the explicit expression of these coefficients seems daunting. One way to prove that a real polynomial has positive coefficients is to establish that it is Hurwitz stable, which means that all its roots have negative real part. Again, some numerical experiments suggest it is the case, as can be seen below in the case $n=20$ using Maple:

\par \medskip
\begin{figure}[ht]
  \includegraphics[width=0.8\linewidth]{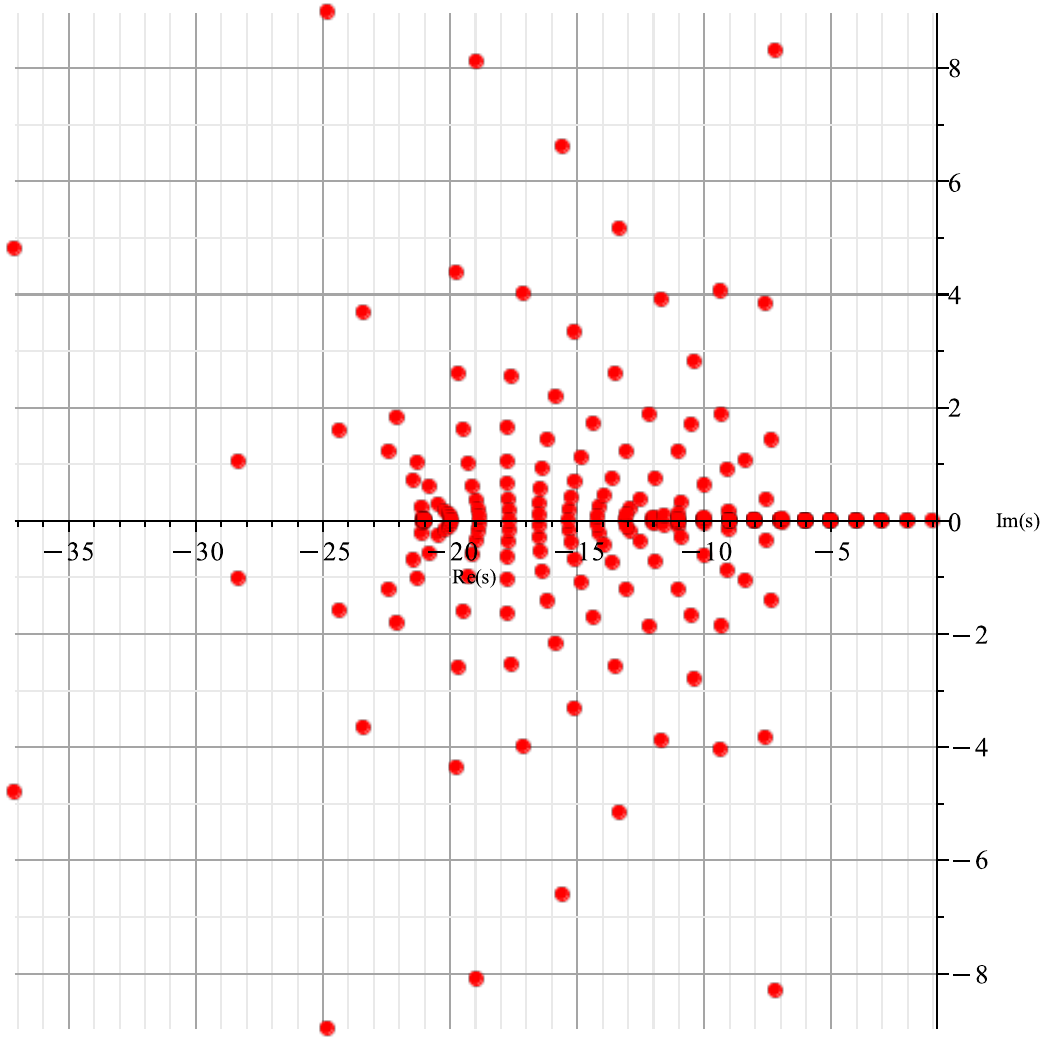}
  \caption{Roots of $\theta_{20}(s)$.}
  \label{fig:racines2}
\end{figure}

An upgrade of the general problem could be that $\theta_n(s)$ is Hurwitz stable for all $s$. Besides, there is a way to include an additional real parameter in the picture using the following result:

\begin{proposition}
The resultant $\mathrm{res}_z(P_n(s,X), P_{n}(s+1, X))$  is equal to 
\[
(-1)^{n(n-1)/2}\frac{s^{(n)}s^{(n-1)} \ldots s^{(1)}}{s^n} \theta_n(s).
\]
\end{proposition}

\begin{proof}
This follows from corollary \ref{yolo} (ii) and proposition \ref{malin} (ii)
\end{proof}

As a corollary, we see that the stability of $\theta_n(s)$ and of the resultant 
\[
\mathrm{res}_z(P_n(s,X), P_{n}(s+1, X))
\] 
are equivalent. It is even possible to extend the conjecture a bit more by decoupling the integers in the resultants. The most general version of the conjecture we propose is:

\begin{conjecture}
For any positive integers $n, m$ such that $n \leq m$, and any $a>0$, the resultant $\mathrm{res}_z(P_n(s, z), P_m(s+a,z))$ is stable.
\end{conjecture}
For $n=m$ and $a=1$, the conjecture boils down to the stability of $\theta_n(s)$. 
\par \medskip
For given values of $n$ and $m$, it is possible to use the Routh algorithm (see \cite[Chap. V, \S 3, Thm. 2]{Gantmacher1959vol2}) to check the veracity of this conjecture. Numerical experiments indicate that the first column of the Routh scheme of the resultant of $P_n(s,z)$ and $P_m(s+a,z)$ are rational functions in $a$ whose numerator and denominator have nonnegative coefficients.
\par \medskip
Some simple cases of the conjecture seem out of reach, for instance $n=1$. It says that for any $a>0$, $P_m(s+a, -s)$ is stable.

\newpage
\bibliographystyle{plain}
\bibliography{biblio}

\begin{thebibliography}{1}

\bibitem{Gantmacher1959vol2}
F.~R. Gantmacher.
\newblock {\em The Theory of Matrices, Vol.~2}.
\newblock Chelsea Publishing Company, New York, 1959.
\newblock Translated from the Russian by K.~A.~Hirsch.

\bibitem{Hilbert1888}
David Hilbert.
\newblock Ueber die discriminante der im endlichen abbrechenden
  hypergeometrischen reihe.
\newblock {\em Journal f\"ur die reine und angewandte Mathematik},
  103:337--345, 1888.

\bibitem{Hurwitz1891}
Adolf Hurwitz.
\newblock Ueber die nullstellen der hypergeometrischen reihe.
\newblock {\em Mathematische Annalen}, 38:452--458, 1891.

\bibitem{Klein1890}
Felix Klein.
\newblock Ueber die nullstellen der hypergeometrischen reihe.
\newblock {\em Mathematische Annalen}, 37:573--590, 1890.

\bibitem{Rahman}
Q.~I. Rahman and G.~Schmeisser.
\newblock {\em Analytic theory of polynomials}, volume~26 of {\em Lond. Math.
  Soc. Monogr., New Ser.}
\newblock Oxford: Oxford University Press, 2002.

\bibitem{Schur1931}
Issai Schur.
\newblock Affektlose gleichungen in der theorie der {L}aguerreschen und
  {H}ermiteschen polynome.
\newblock {\em Journal f\"ur die reine und angewandte Mathematik}, 165:52--58,
  1931.

\bibitem{Szego1975}
G{\'a}bor Szeg{\H{o}}.
\newblock {\em Orthogonal Polynomials}, volume~23 of {\em American Mathematical
  Society Colloquium Publications}.
\newblock American Mathematical Society, Providence, RI, 4 edition, 1975.

\end{thebibliography}

\end{document}